\documentclass[12pt]{extarticle}
\usepackage{amsmath, amsthm, amssymb, color}
\usepackage[backref=page]{hyperref}
\usepackage{graphicx}
\usepackage[all]{xypic}
\usepackage{verbatim}
\usepackage{tikz}
\usepackage{booktabs}
\usepackage{cite}
\usepackage{caption}
\usepackage{subcaption}
\usepackage{listings}
\usepackage{cleveref}
\usepackage[format=plain, font=footnotesize]{caption}
\usepackage{enumitem}
\oddsidemargin \evensidemargin
\definecolor{turqouise}{rgb}{0.07, 0.41, 0.35}
\definecolor{linkblue}{rgb}{0.1, 0.46, 0.82}
\hypersetup{
  colorlinks=true,
  citecolor=turqouise,
  linkcolor=linkblue,
  urlcolor=black
}

\usepackage{makecell}
\usepackage{array}
\newcolumntype{?}{!{\vrule width 1pt}}

\graphicspath{ {figures/} }

\newtheorem{theorem}{Theorem}
\numberwithin{theorem}{section}
\newtheorem{proposition}[theorem]{Proposition}
\newtheorem{lemma}[theorem]{Lemma}

\newtheorem{definition}[theorem]{Definition}

\newtheorem{conjecture}[theorem]{Conjecture}
\theoremstyle{definition}

\newtheorem{remark}[theorem]{Remark}

\makeatletter
\newcommand{\subjclass}[2][1991]{%
  \let\@oldtitle\@title%
  \gdef\@title{\@oldtitle\footnotetext{#1 \emph{Mathematics Subject Classification.} #2}}%
}
\newcommand{\keywords}[1]{%
  \let\@@oldtitle\@title%
  \gdef\@title{\@@oldtitle\footnotetext{\emph{Key words and phrases.} #1.}}%
}
\makeatother

\date{}

\title{\textbf{Critical Curvature of Algebraic Surfaces\\ in Three-Space}
\author{Paul Breiding, Kristian Ranestad, and
Madeleine Weinstein}}


\begin{document}

\maketitle
\setcounter{page}{1}

\begin{abstract}
 \noindent
We study the curvature of a smooth algebraic surface $X\subset \mathbb R^3$ of degree $d$ from the point of view of algebraic geometry. More precisely, we consider umbilical points and points of critical curvature. We prove that the number of complex critical curvature points is of order $d^3$. For general quadrics, we fully characterize the number of real and complex umbilics and critical curvature points.
 \end{abstract}

\section{Introduction}

Curvature is one of the most significant objects of study in differential geometry. This paper takes an unusual approach to curvature in that we study it from the perspective of enumerative geometry. We \emph{count} points where the curvature exhibits a special property. 

For the purpose of illustration, we first consider the case of plane curves. At each point on a plane curve, there is a best-fit circle called the \textit{osculating circle}. The radius of the osculating circle is called the \textit{radius of curvature}. Points on a curve where the radius of curvature achieves a local maximum or minimum are called \textit{critical curvature points}. In work published in 1865 \cite{Sal65}, George Salmon showed that a general algebraic plane curve of degree $d$ has $6d^2-10d$ complex critical curvature points. 

In this article, we attempt to extend Salmon's work to a general surface 
$$X=\{x\in\mathbb R^3 \mid f(x)=0\} \subset \mathbb R^3,$$
where $f(x)$ is a \emph{general} polynomial in three variables $x=(x_1,x_2,x_3)$ of degree $d$. 

Every point on a smooth surface in $\mathbb{R}^3$ has two principal curvatures. 
Similar to the case of curves, we call a point on $X$ where one of the two principal curvatures achieves a critical value a \emph{critical curvature point}. We also consider a second type of point: umbilics. A point where the two principal curvatures of $X$ are equal is called an \emph{umbilical point}. At such a point, the best second-order approximation of the surface is a sphere. 
Our first main result is that all umbilics are critical curvature points. We prove the next theorem in \Cref{umbilical_are_cc_proof}.
\begin{theorem}\label{umbilical_are_cc}
Let $x$ be a umbilical point of $X\subset \mathbb R^3$. Then $x$ is also an isolated critical curvature point of $X$. 
\end{theorem}
This result implies that the number of umbilical points is a lower bound for the number of critical curvature points. The former was computed by Salmon in \cite{Sal65} for general surfaces. Let us recall his result.
\begin{theorem}[Salmon \cite{Sal65}]\label{salmons_theorem} 
A general surface of degree $d$ in $\mathbb{R}^3$ has finitely many umbilics. The number of complex umbilical points is $10d^3-28d^2+22d$.
\end{theorem}
By complex umbilical points we mean the following. In Section \ref{sec:curvature_algebraic}, we derive polynomial equations for umbilical points. Their complex zeros are complex umbilical points  (see Definition \ref{def_complex_points}). We define complex critical curvature points similarly. 
In \Cref{thm:curvatureupperbound}, we prove that
$\frac{699}{2}d^3-\frac{1611}{2}d^2+462d$ is an upper bound for the number of isolated complex critical curvature points of $X$. We conclude that the number of complex critical curvature points grows as $d^3$. 
\begin{theorem}\label{cor_growth}
  Let $X \subset \mathbb{R}^3$ be a general algebraic surface of degree $d$. Then the number $N(d)$ of isolated complex critical curvature points grows as $d^3$. More precisely, 
  $$10\leq \liminf_{d\to\infty}\, \frac{N(d)}{d^3} \quad \text{and}\quad \limsup_{d\to\infty}\, \frac{N(d)}{d^3} \leq \frac{699}{2}.$$
\end{theorem}

\Cref{sec:quadrics} characterizes the number of real and complex umbilical points and critical curvature points in the special case of quadric surfaces. We prove in \Cref{thm_quadrics_umbilics} that a general quadric surface in $\mathbb R^3$ has 12 umbilical points and either 0 or 4 of them are real. More specifically, a general ellipsoid or two-sheeted hyperboloid has 4 real umbilics, while a one-sheeted hyperboloid has no real umbilics. Examples of this can be seen in \Cref{fig1}.  Furthermore, we prove in \Cref{thm_cc} that a general quadric surface in $\mathbb R^3$ has 18 complex critical curvaure points. Here, an ellipsoid has 10, a one-sheeted hyperboloid has 4, and a two-sheeted hyperboloid has 6 real critical curvature points. The case of an ellipsoid recovers results in \cite{Gangopadhyay}. In particular, this shows that neither all complex umbilical points nor all complex critical curvature points can be real. We suspect the same is true for surfaces of higher degree.
\begin{conjecture}
Neither the complex umbilics nor the complex critical curvature points of a general surface of degree $d\geq 3$ in $\mathbb{R}^3$ can all be real. 
\end{conjecture}
This is reminiscent of the classical result of Klein \cite{Klein1876} that a general curve in~$\mathbb R^2$ of degree~$d$ has $3d(d-2)$ complex inflection points, but that not all of them can be real (in fact, at most a third of the complex inflection points can be real). 

The rest of the paper is organized as follows: in the next section, we formulate polynomial equations for principal curvatures. In \Cref{sec:crit_curv}, we compute an upper bound for the number of complex critical curvature points.  In \Cref{sec:quadrics}, we compute umbilics and critical curvature explicitly for quadrics in $\mathbb R^3$. 

Previous work on counting curvature points includes \cite{VCMAGPC, piene2021return,  Gangopadhyay, Emil, Kazarian, KUV}. The papers \cite{VCMAGPC} and \cite{piene2021return} address critical curvature in the case of plane curves. In \cite{VCMAGPC}, the problem is approached from a computational perspective. The paper \cite{piene2021return} describes the correspondence between the critical curvature points on a plane curve and the cusps on its evolute. The discussion around \cite[Problem 2]{piene2021return} mentions that the possible numbers of real critical curvature points on a plane curve are not known. The paper \cite{Emil} describes an algorithm to compute the critical curvature points and their degree of a higher dimensional variety. The local geometry of umbilical points on surfaces in $\mathbb{R}^3$ is classified by Gautam Gangopadhyay \cite{Gangopadhyay}.

\begin{figure}
\begin{center}
\includegraphics[width = 0.3\textwidth]{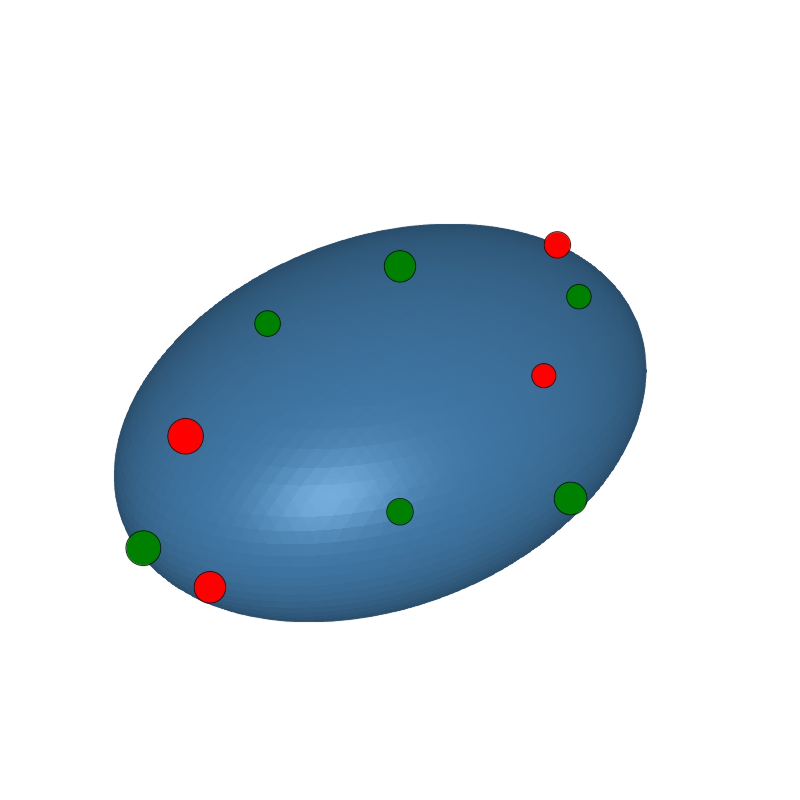}
\includegraphics[width = 0.3\textwidth]{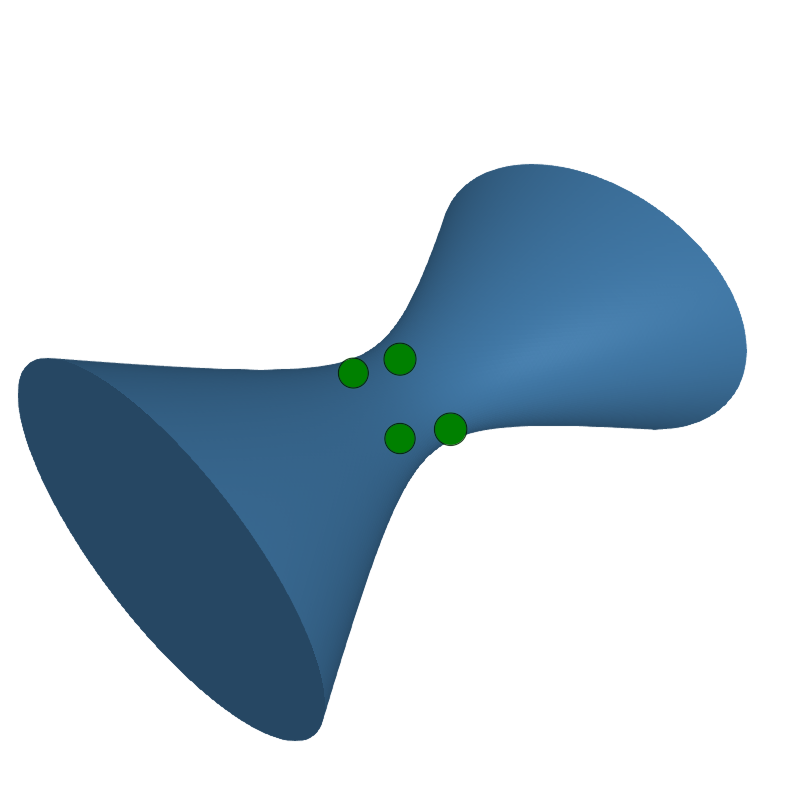}
\includegraphics[width = 0.3\textwidth]{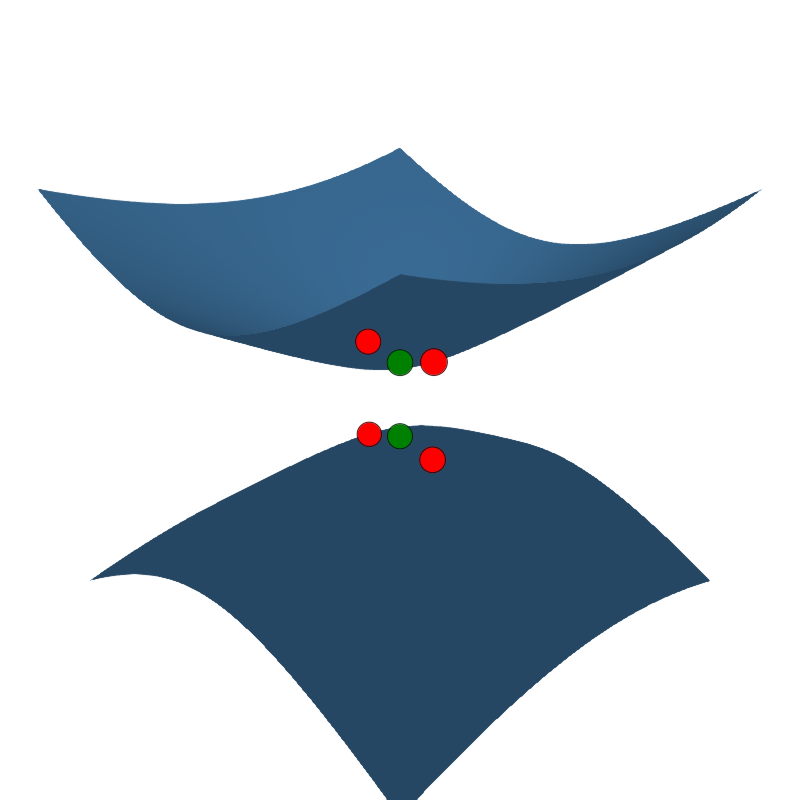}
\end{center}
\caption{\label{fig1} The pictures show the three quadric surfaces $X_1=\{x_1^2 + 2x_2^2 + 4x_3^2= 1\}$ (ellipsoid, left picture) and $X_2 = \{-x_1^2 + 2x_2^2 + 4x_3^2= 1\}$ (one-sheeted hyperboloid, picture in the middle) and $X_3=\{-2x_1^2 - x_2^2 + 4x_3^2= 1\}$ (two-sheeted hyperboloid, right picture). The umbilical points are shown in red and the critical curvature points are shown in green. All these points have (at least) one coordinate equal to zero. This is not a coincidence, but follows from \Cref{distinctai}.
The ellipsoid has~4 real umbilics and 10 real critical curvature points,
the one-sheeted hyperboloid has no real umbilics and 4 real critical curvature points, and the two-sheeted hyperboloid has 4 real umbilics, and 6 real critical curvature points. The pictures were created using the software packages \texttt{Makie.jl} \cite{Makie} and \texttt{Implicit3DPlotting.jl} \cite{Himmelmann2021}.}
\end{figure}

\bigskip
\section{Curvature of an Algebraic Hypersurface}\label{sec:diff_geo}
We give a brief introduction to the curvature of algebraic hypersurfaces. For more details see the textbooks \cite{Lee, doCarmo} and, in particular, \cite[Chapter 6]{BKS}.

We consider a general hypersurface $X=\{x\in\mathbb R^n \mid f(x)=0\}$ of degree~$d$. The relevant case for our paper is $n=3$. 
We assume that $f$ has no complex singular zeros. This means that the system of $n+1$ polynomial equations $f(x) = \tfrac{\partial f(x)}{\partial x_1} = \cdots = \tfrac{\partial f(x)}{\partial x_n} =0$ in $n$ variables $x=(x_1,\ldots,x_n)$ has no solution in~$\mathbb C^n$. We denote the gradient $\nabla f(x) = (\tfrac{\partial f(x)}{\partial x_1},\ldots, \tfrac{\partial f(x)}{\partial x_n})^T$.

We denote the Euclidean inner product by 
$$\langle v,w\rangle := v^Tw.$$
The Euclidean norm is $\Vert v\Vert = \sqrt{\langle v,v\rangle}$.

Let $p\in X$ be a point and $v\in T_pX$, $\Vert v\Vert = 1$, be a tangent vector of unit norm. A unit normal vector of $X$ at $p$ is given by the normalized gradient $\Vert \nabla f(p)\Vert^{-1}\cdot \nabla f(p)$. The \emph{curvature} $c(p,v)$ of $X$ at $p$ in direction~$v$ is then the $v$-component of the directional derivative in direction $v$ of this normalized gradient. That is,
$$
c(p,v) = \left\langle v, \frac{\partial (\Vert \nabla f\Vert^{-1}\cdot \nabla f}{\partial v}(p) \right\rangle.
$$
The curvature can be expressed via the \emph{Hessian} of $f$ at $p$:
\begin{equation}\label{curvature_equation2}
c(p,v)= \frac{v^T\, \mathrm{H}f(p)\, v}{\Vert \nabla f(p)\Vert},
\end{equation}
where 
$$\mathrm{H}f(p) := \begin{bmatrix}
\frac{\partial^2 f}{\partial x_1^2}(p) & \cdots & \frac{\partial^2 f}{\partial x_1\partial x_n}(p)\\
\vdots  &\ddots &\vdots\\
\frac{\partial^2 f}{\partial x_1\partial x_n}(p) & \cdots & \frac{\partial^2 f}{\partial x_n^2}(p)
\end{bmatrix};$$
see, e.g., \cite[Example 6.14]{BKS}. For a fixed $p$, the right-hand side in  (\ref{curvature_equation2}) is a quadratic form $T_pX\times T_pX\to\mathbb R$, called the \emph{second fundamental form} of $X$ at $p$. Its eigenvalues are called the \emph{principal curvatures} of $X$ at $p$. The corresponding eigenvectors are called \emph{principal curvature eigenvectors}.

\subsection{Polynomial Equations for Principal Curvatures}\label{sec:curvature_algebraic}
We now give polynomial equations for the principal curvatures of $X$. A principal curvature eigenvector $u$ is an eigenvector of the 
quadratic form $T_pX\times T_pX\to\mathbb R$ given by~(\ref{curvature_equation2}) with eigenvalue $y_1$. We orthogonally decompose the vector $u$, separating out the normal component and thus obtaining the equation
$$
\mathrm{H}f(p) u + y_1 u + y_2 \nabla f(p)=0,$$
where $y_1,y_2\in\mathbb R$ are scalars. The curvature $c(p,v)$ involves the square root of 
\begin{equation}\label{def_eta}\eta(p) := \langle \nabla f(p),\nabla f(p)\rangle.\end{equation}
To obtain polynomial equations, we introduce a new variable~$\lambda$ and the polynomial equation  $\lambda^2 \eta(p)=1$. Then, by (\ref{curvature_equation2}) the curvature of $X$ at $p$ in the direction $v\in T_pX$ is given by~$\pm \lambda \, (v^T\mathrm H(p)v)$. The ambiguity of the sign does not concern us because principal curvatures are defined up to sign.

We thus have the following list of $n+3$ equations in the $2n+3$ variables $\{x_1, \dots, x_n, \lambda, u_1, \dots, u_n, y_1,y_2\}$ describing principal curvatures:
\begin{equation}\label{crit_curv_equations}
\begin{bmatrix}
\lambda^2 \eta(x)-1\\[0.3em]
\langle u, \nabla f(x)\rangle \\[0.3em]
\langle u,u\rangle -1=0\\[0.3em]
\, \mathrm{H}f(x) u + y_1u+y_2\nabla f(x)\,
\end{bmatrix} = 0.
\end{equation}
The second equation establishes that $u\in T_xX$ and the third equation that $u$ has norm $1$ (as eigenvectors are defined only up to scaling). We denote
\begin{equation}\label{def_C}
\mathcal C(f):=\{(x,\lambda, u, y)\in X \times\mathbb R\times \mathbb R^n \times \mathbb R^2 \mid (x,\lambda, u, y)\text{ satisfies the equations in (\ref{crit_curv_equations})}\}.
\end{equation}
This is a real algebraic variety, which we call the \emph{curvature variety} of $X$. We have the following observation.
\begin{proposition}\label{curv_proposition}
Let $\pi: \mathcal C(f) \to X, (x,\lambda,u,y)\mapsto x$ be the projection onto the first coordinate. For any fixed $x\in X$, the points in the preimage $(x,\lambda, u, y)\in\pi^{-1}(x)$ are points such that $u\in T_xX$ is a principal curvature eigenvector of $X$ at $x$. Its principal curvature is 
$$g(x,\lambda,u,y):= \lambda \, u^T \mathrm{H}f(x) u.$$
Equivalently,
$$g(x,\lambda,u,y) = -\lambda \, y_1.$$
\end{proposition}
\begin{proof}
The $u$ in $(x,\lambda, u, y)\in\pi^{-1}(x)$ is a principal curvature eigenvector at $x\in X$ by definition. By (\ref{curvature_equation2}), the curvature in direction $u$ is given by
$$c(x,u)=\frac{u^T \mathrm{H}f(x) u}{\sqrt{\eta(x)}}.$$ We combine this with $\lambda^2 = 1/\eta(x)$.
Finally, multiplying the last entry in \eqref{crit_curv_equations} from the left by $u^T$ shows that~$u^T \mathrm{H}f(x) u = -y_1$.
\end{proof}

Umbilical points of $X$ are points where two of the principal curvatures coincide. Moreover, critical curvature points of $X$ are points where one of the principal curvatures attains a critical value. 
We now give an algebraic definition using the curvature variety of $X$.
\begin{definition}[Algebraic Definition of Umbilical and Critical Curvature Points]\label{def_points}
Let $X$ be a smooth algebraic hypersurface in $\mathbb R^n$ defined by a polynomial $f$. Let $\mathcal C(f)$ be its curvature variety and denote the projection $\pi:\mathcal C(f)\to X$. A point $x\in X$ is called
\begin{enumerate}
\item an {umbilical point} if there exists $(x,\lambda,u,y), (x,\lambda',u',y')\in \pi^{-1}(x)$ with $u' \neq \pm u$ such that $y_1 = y_1'$;
\item a {critical curvature point} if there exists $(x,\lambda,u,y)\in \pi^{-1}(x)$ such that $(x,\lambda,u,y)$ is a critical point of $g$.
\end{enumerate}
\end{definition} 
Let $\mathcal P_{n,d}$ be the vector space of real polynomials in $n$ variables of degree at most $d$.
We collect umbilics and critical curvature points of all smooth hypersurfaces in
$$U_{n,d}:=\bigcup_{f\in\mathcal P_{n,d}\, :\, \text{Sing}(f) = \emptyset} \{f\}\times U(f) \quad \text{and} \quad Cr_{n,d}:=\bigcup_{f\in\mathcal P_{n,d}\, :\, \text{Sing}(f) = \emptyset} \{f\}\times Cr(f);$$
i.e.,  $U_{n,d}$ and $Cr_{n,d}$ are semialgebraic sets of $\mathcal P_{n,d}\times \mathbb R^n$.

Next, we define complex umbilical and complex critical curvature points as subvarieties of the complex Zariski closures  \cite{whitney}  of $U_{n,d}$ and $Cr_{n,d}$.
\begin{definition}[Complex Umbilical and Critical Curvature Points]\label{def_complex_points}
We denote the complex Zariski closures of the real algebraic varieties $U_{n,d}$ and $Cr_{n,d}$ by $\overline{U_{n,d}}$ and $\overline{Cr_{n,d}}$.
\begin{enumerate}
\item For a fixed hypersurface $X=\{f=0\}$ and every $(f,x)\in\overline{U_{n,d}}$, we call $x$ a complex umbilical point of $X$.
\item For a fixed hypersurface $X=\{f=0\}$ and every $(f,x)\in\overline{Cr_{n,d}}$, we call $x$ a complex critical curvature point of $X$.
\end{enumerate}
\end{definition}
The two definitions might seem unnecessarily heavy at first sight. Let us walk slowly through their components to see why they are reasonable.

First, the reason why we do not simply define complex umbilical or critical curvature points by replacing $\mathbb R$ by $\mathbb C$ in the definition of $\mathcal C(f)$ is as follows. Real umbilical points are points where the second fundamental form has a repeated eigenvalue. However, the matrix discriminant of symmetric matrices is a sum-of-squares (see, e.g., \cite{Ilyushechkin2005, Parlett2002}) and so the real variety of matrices with repeated eigenvalues is contained in the singular locus defined by the matrix discriminant. We will see this explicitly, when we work out the case of quadrics in Section \ref{sec:quadrics}. Therefore, if we simply replaced $\mathbb R$ by $\mathbb C$ in the definition of $\mathcal C(f)$, we would obtain a variety in which the real umbilical points are singular. To avoid this, we first define the real locus and then take the complex Zariski closure.

\begin{remark}\label{remark_cc}
In the case of complex critical curvature points, the above approach is not necessary. Instead, we can consider the system of polynomial equations, where the Jacobian of the system defining the curvature variety and the gradient of $g$ are linearly dependent; i.e. where the $(2n+3)\times (n+5)$-matrix $A$ in \eqref{def_A} below has rank at most $n+4$. 
\end{remark}

Secondly, before taking the complex Zariski closure, we take the union of $U_{n,d}(f)$ and $Cr_{n,d}(f)$ over all~$f$. The reason for this is that real hypersurfaces can potentially have no umbilical points or critical curvature points. For instance, the real locus of the hypersurface could be empty (consider the hypersurface defined by~$x_1^2 + x_2^2 + 1$). To get a well-defined count over the complex numbers for general polynomials, we have to define complex umbilical and complex critical curvature points also for complex hypersurfaces without real umbilical or real critical curvature points.

\subsection{Umbilics are Critical Curvature Points }\label{umbilical_are_cc_proof}
We now focus on the case $n=3$; i.e., surfaces in three-space. The first step for proving 
\Cref{umbilical_are_cc} is the following lemma.

\begin{lemma}\label{lem_smooth}
In the case $n=3$, for general $f$ the curvature variety $\mathcal C(f)$ is smooth and of dimension $2$. 
\end{lemma}
\begin{proof}
Let $(x,\lambda,u,y)\in \mathcal C(f)$. We have 
$$\lambda = \pm (\eta(x))^{-1/2},\quad y_1 = -u^T\mathrm H f(x)u, \quad y_2 = -\frac{\nabla f(x)^T\mathrm H f(x)u}{ \eta(x)}.$$
Thus, $\lambda$ and $y$ are smooth functions in $x$ and $u$. If $x$ is not an umbilical point, $u$ is a smooth function of $x$. In this case, locally around $(x,\lambda,u,y)$  we have that $\mathcal C(f)$ is smooth and of dimension $2$. Now, assume that $x$ is an umbilical point. We have the smooth curve $(x,\lambda, u(t),y(t))\in\mathcal C(f)$, where $u(t)$ parametrizes the unit circle in $T_xX$. On the other hand, if we have a smooth curve $x(t)$ in $X$ with $x(0)=x$ that is not constant, we have just shown that this induces a curve $\gamma(t)=(x(t),\lambda(t),u(t),y(t))\in\mathcal C(f)$ with $\lim_{t\to 0} (x(t),\lambda(t),u(t),y(t)) = (x,\lambda,u,y)$ that is smooth for $t\neq 0$. 
Let $v:=\lim_{t\to 0 } \frac{\partial u(t)}{\partial t}$, then
by \cite[Proposition 12.3]{porteous}, $\lim_{t\to 0} \frac{\partial x(t)}{\partial t}$ is uniquely determined by $v$. This shows that $\gamma(t)$ is in fact smooth. Therefore, $\mathcal C(f)$ is also smooth at umbilical points.
\end{proof}

\begin{proof}[Proof of Theorem \ref{umbilical_are_cc}]
Let $x\in X=\{f=0\}$ be an umbilical point of $X$. We prove that $x$ is also a critical curvature point. 

In the tangent plane $H:=T_xX$ every $u\in H$ is a principal curvature eigenvector of $X$ at~$x$. Let $S^1$ be the unit circle in $H$. We construct a function $\psi:S^1\to\mathbb R$ as follows. 
For every unit tangent vector $v\in S^1$ we can find a unique (locally defined) geodesic $x(t)\in X$ with $\frac{\partial x(t)}{\partial t}|_{t=0}=v$. 
In the proof of \Cref{lem_smooth}, we have shown that this uniquely defines a smooth curve 
$\alpha(t):=(x(t),\lambda(t),u(t),y(t))\in\mathcal C(f)$. Let us write $x_t:=x(t)$ and $u_t:=u(t)$. Recall from \eqref{curvature_equation2} that the curvature given by $u_t$ along $x_t$
is
$$c(x_t,u_t) =\frac{u_t^T \mathrm Hf(x_t) u_t}{\Vert \nabla f(x_t)\Vert} 
.$$
Then, denoting by $\mathrm D^3f(x)$ the tensor of third-order derivatives of $f$ at $x$, and by $\dot u$ the derivative of $u_t$ at $t=0$ we have 
$$\frac{\partial c(x_t,u_t)}{\partial t}\Bigm|_{t=0} = \frac{2u^T \mathrm Hf(x) \dot u + \mathrm D^3f(x)(v, u,u)}{\Vert \nabla f(x)\Vert} + (u^T \mathrm Hf(x) u)\, \Big(-\frac{\partial}{\partial t} \Vert \nabla f(x_t)\Vert^{-1}\Big)\Big|_{t=0}.$$
We define
$$\psi(v):=\frac{\partial c(x_t,u_t)}{\partial t} \Bigm|_{t=0}.$$
If $\psi(u)=0$, we have that $x$ is a critical curvature point because the tangent space of $\mathcal C(f)$ is spanned by $\frac{\partial \alpha}{\partial t}|_{t=0}$ and the derivative of $(x,\lambda, u(t),y(t))$ with $u(t)\in S^1$, along which the curvature is constant. 

If $\psi(v)\neq 0$, we consider $\psi(-v)$, which is given by the curve $t\mapsto \alpha(-t)$. We have 
\begin{align*}
\psi(-v) &= \frac{2u^T \mathrm Hf(x) (-\dot u) + \mathrm D^3f(x)(-v, u,u)}{\Vert \nabla f(x)\Vert} + (u^T \mathrm Hf(x) u)\, \Big(-\frac{\partial}{\partial t} \Vert \nabla f(x_{-t})\Vert^{-1}\Big)\Big|_{t=0}\\
  & =-\psi(v).
  \end{align*}
Since $\psi$ is continuous, by the intermediate value theorem there must exist $w\in S^1$ with $\psi(w)=0$. Consequently, $x$ is a critical curvature point.

Finally, we prove that every umbilical point is an isolated critical curvature point. For this, we write $f(x_1,x_2,x_3)=0$  in Monge form locally around the umbilic $x=(0,0,0)$ (see \cite[Section 12.4]{porteous}):
$$x_3 = \frac{c}{2}\cdot (x_1^2+x_2^2) + (px_1^3 + qx_1^2x_2 + r x_1x_2^2 + sx_2^3) + O(x_1,x_2)^4.$$
Here, $c$ is the curvature at $x$ and $p,q,r,s$ are scalars depending on $f$. In \cite[Section 12.4]{porteous} it is shown that the second fundamental form around $x$ is given by the matrix 
$$\begin{pmatrix}c & 0 \\ 0 & c\end{pmatrix}  
+ x_1
\begin{pmatrix}p & q \\ q & r\end{pmatrix}
+ x_2\begin{pmatrix}q &r \\ r & s\end{pmatrix}
.$$
By Salmon's Theorem (\Cref{salmons_theorem}), a general $f\in \mathcal P_{n,d}$ has a finite number of umbilical points, so we fix one $f_0$, and a small real disk $D\subset \mathcal P_{n,d}$ centered at $f_0$, such that every $f\in D$ has the same finite number of real umbilical points.
 Over $D$, the map $U_{n,d} \to \mathcal P_{n,d}$ is finite and etale. Consider one of the umbilical points say $x_0\in \{f_0=0\}$, and the sheet $D_0\subset U_{n,d}$ over $D$ that it belongs to.  For $t \in D$, we get a $f_t$ and an umbilical point~$x_t\in\{f_t=0\}$.  
 Consider a tubular neighborhood of $D_0$, i.e.\ a neighborhood $\Delta_t$ in $\{f_t=0\}$ of $x_t$ for each~$t$. Using the implicit function theorem, this tubular neighborhood may be chosen to be the image of a map $${\rm id}_D\times\phi(t,s): D\times \Delta\to D\times\cup_{t\in D}\Delta_t; \; (t,s)\mapsto (t,\phi(t,s))\in t\times \Delta_t,$$ where $\Delta$ is a $2$-dimensional disc with parameter $s$. Then all entries in the fundamental forms on  $\Delta_t$ are functions in $t$ and $s$.  Now genericity in $t$ induces a genericity in fundamental forms and eigenvalues, so for a general $t\in D$ no point in  $\Delta_t$ is a critical curvature point. Since there is a finite number of umbilical points on $V(f_0)$ one can apply this argument to each umbilical point and find a $t$ such that $V(f_t)$ has no critical curvature points in the neighborhood of any its umbilical points.
 
\end{proof}

\bigskip
\section{An Upper Bound for Critical Curvature Points}\label{sec:crit_curv}
We study critical curvature points for surfaces in $\mathbb R^3$. Recall from (\ref{def_C}) the definition of the variety $\mathcal C(f)$ and from \Cref{curv_proposition} the curvature polynomial $g(x,\lambda,u,y)$. 

Salmon's Theorem (\Cref{salmons_theorem}) says that a general surface of degree $d$ in $\mathbb R^3$ has finitely many umbilical points. This implies using \Cref{umbilical_are_cc} that a general surface of degree $d$ has isolated critical curvature points, so the complex variety $\overline{Cr_{n,d}}$ from \Cref{def_complex_points} has a zero-dimensional component. The degree of this component is a well-defined number, and we can get an upper bound for the degree using Porteous' formula; see, e.g., \cite[Chapter 12]{3264}. This leads to the following result.

\begin{theorem}\label{thm:curvatureupperbound}
Let $X \subset \mathbb{R}^3$ be a general algebraic surface of degree $d$. Then, $X$ has isolated complex critical curvature points of $X$. An upper bound for their number is given by
$$\frac{699}{2}d^3-\frac{1611}{2}d^2+462d.$$
\end{theorem}
We prove the theorem at the end of the section.

Let us first remark that the formula is stated in the case $n=3$, but can be computed for any $n \ge 3$ using the same process that we give in our proof below. However, for $n\geq 3$ we do not have a proof that a general hypersurface always has isolated complex critical curvature points. Experiments suggest the following conjecture.

\begin{conjecture}\label{conj_isolated_cc}
Let $n\geq 4$ and  $X\subset \mathbb R^n$ be a general algebraic hypersurface of degree $d$. Then, $X$ has isolated complex critical curvature points.
\end{conjecture}

Note that for $n\geq 4$ hypersurfaces in~$\mathbb R^n$ always have a positive-dimensional subvariety of umbilical points. This follows from a simple count of dimensions: the codimension of $(n-1)\times (n-1)$ real symmetric matrices with repeated eigenvalues is equal to~$2$ (see, e.g., \cite[Lemma 1.1]{Arnold1972}) and for~$n\geq 4$ we have that~$2 < n-1$.

\Cref{thm:curvatureupperbound} provides an upper bound for the degree of the critical curvature locus rather than an exact formula. Evaluating our formula for small values of $d$ gives

\smallskip
\begin{center}
\begin{tabular}{|l|l|l|}
\hline
$d$ & $\tfrac{699}{2}d^3-\tfrac{1611}{2}d^2+462d$ & number of complex critical curvature points\\[0.25em]
\hline
2 & 498 & $18$\\
3&  3573 & $\geq 456$\\
4& 11328 & $\geq 1808$\\
\hline
\end{tabular}
\end{center}
\smallskip

The last column is established as follows: the first entry $18$ follows from \Cref{thm_cc}. This number is far from the number 498, which we get from our formula. For $d=3$ and $d=4$, we solve the critical curvature equations using HomotopyContinuation.jl \cite{BT2018} (the code is attached to the \texttt{arXiv} version of this article). We certify\footnote{Certification provides a computer-assisted proof that the number of complex critical curvature points of a general surface of degree $d=3$ in three-space is at least 456. Similarly, for $d=4$ and $1808$ critical curvature points.} \cite{BRT2020} $456$ critical curvature points for $d=3$ and $1808$ critical curvature points for $d=4$, which is again far from the values implied by our formula. This gap is not surprising given that we are using counting formulas for general equations while the critical curvature equations are highly structured.

Now, we prove \Cref{thm:curvatureupperbound}.

\begin{proof}[Proof of \Cref{thm:curvatureupperbound}]
Salmon's Theorem (\Cref{salmons_theorem}) and \Cref{umbilical_are_cc} together imply that $X$ has isolated complex critical curvature points.

Let us now consider the equations for complex critical curvature points for general $n$. Later we specialize to $n=3$.
We introduce the following shorthand notations for the polynomials in (\ref{crit_curv_equations}):
$$\alpha = \lambda^2(\nabla f\cdot \nabla f)-1, \quad \beta = \langle u, \nabla f(x)\rangle, \quad v =
\sum_{i=1}^{n} u_i^2-1, \quad w =
\mathrm{H}f(x) u + y_1u+y_2\nabla f.$$
Recall from \Cref{curv_proposition} that the curvature is given by
$$g(x,\lambda, u,y) = \lambda u^T \mathrm{H}f(x) u.$$
Following \Cref{remark_cc}, critical curvature points are images of the critical points of the function $g$ on the curvature variety, so we estimate the cardinality of the former by a count of the latter.  These are points $(x,\lambda,u,y)$ where the Jacobian of the system defining the curvature variety and the gradient of $g$ are linear dependent, i.e. where the following $(2n+3)\times (n+5)$-matrix $A$ of partial derivatives has rank at most $n+4$:

\begin{equation}\label{def_A}
A=\begin{pmatrix}
f_{x_1}&\alpha_{x_1}&\beta_{x_1}&0&w_{1,x_1}&\ldots&w_{n,x_1}&g_{x_1}\\
\vdots&\vdots&\vdots&\vdots&\vdots&\vdots&\vdots&\vdots\\
f_{x_n}&\alpha_{x_n}&\beta_{x_n}&0&w_{1,x_n}&\ldots&w_{n,x_n}&g_{x_n}\\
0&0&\beta_{u_1}&u_1&w_{1,u_1}&\ldots&w_{n,u_1}&g_{u_1}\\
\vdots&\vdots&\vdots&\vdots&\vdots&\vdots&\vdots&\vdots\\
0&0&\beta_{u_n}&u_n&w_{1,u_n}&\ldots&w_{n,u_n}&g_{u_n}\\
0&\alpha_{\lambda}&0&0&0&\ldots&0&g_{\lambda}\\
0&0&0&0&w_{1,y_1}&\ldots&w_{n,y_1}&0\\
0&0&0&0&w_{1,y_2}&\ldots&w_{n,y_2}&0\\
\end{pmatrix}.
\end{equation}
As explained in the discussion of Porteous' formula in \cite[Chapter 12]{3264}, this locus has codimension $n-1$, which is equal to the dimension of $X$. Thus, we can use this system of polynomial equations to obtain an upper bound for the number of isolated critical curvature points.
We pass to complex projective space: we now denote by
$${\mathbb P}\mathcal C(f)\subset \mathbb{P}_{x}^n\times\mathbb{P}_{u}^n\times\mathbb{P}_{\lambda}^1\times\mathbb{P}_{y}^2$$ the complex projective closure of the curvature variety $\mathcal C(f)$ from \eqref{def_C}.
This introduces homogenizing variables $x_0, u_0,\lambda_0$ and $y_0$. When we homogenize the entries of $A$, the multidegrees are given by the following matrix:
\smallskip

{\scriptsize
\[
\begin{pmatrix}
  \ (d-1,0,0,0)&(2d-3,0,1,0)&(d-2,1,0,0)&0&(d-2,1,0,1)&\ldots&(d-2,1,0,1)&(d-3,2,1,0)\ \\
&&&&&&&\\
\ (d-1,0,0,0)&(2d-3,0,1,0)&(d-2,1,0,0)&0&(d-2,1,0,1)&\ldots&(d-2,1,0,1)&(d-3,2,1,0)\ \\
0&0&(d-1,0,0,0)&(0,1,0,0)&(d-1,0,0,1)&\ldots&(d-1,0,0,1)&(d-2,1,1,0)\ \\
\vdots&\vdots&\vdots&\vdots&\vdots&\vdots&\vdots&\vdots\\
0&0&(d-1,0,0,0)&(0,1,0,0)&(d-1,0,0,1)&\ldots&(d-1,0,0,1)&(d-2,1,1,0)\ \\
0&(2d-2,0,0,0)&0&0&0&\ldots&0&(d-2,2,0,0)\ \\
0&0&0&0&(d-1,1,0,0)&\ldots&(d-1,1,0,0)&0\\
0&0&0&0&(d-1,1,0,0)&\ldots&(d-1,1,0,0)&0\\
\end{pmatrix}.
\]}
\smallskip

A matrix that defines a map $\mathcal{O}(a) \oplus \mathcal{O}(b) \to \mathcal{O}(c) \oplus \mathcal{O}(d)$ has entries of degrees 

\[ 
\begin{pmatrix}
    c-a & d-a  \\
     c-b & d-b \\ 
\end{pmatrix}. 
\] 

Thus $A$ is the matrix of the map of vector bundles $B_1\to B_2$ with
\begin{align*}
B_1 :=\; &n{\mathcal O}(1,0,0,0)\oplus n{\mathcal O}(0,1,0,0)\oplus{\mathcal O}(0,0,1,0)\oplus 2{\mathcal O}(0,0,0,1) \\
B_2 :=\; &{\mathcal O}(d,0,0,0)\oplus{\mathcal O}(d-1,1,0,0) \oplus{\mathcal O}(0,2,0,0)
\oplus{\mathcal O}(2d-2,0,1,0)\\
&\oplus n{\mathcal O}(d-1,1,0,1)\oplus{\mathcal O}(d-2,2,1,0).
\end{align*}
We use Macaulay2 \cite{M2} to evaluate Porteous' formula \cite[Theorem 12.4]{3264} in this case. We obtain the class of the locus where the matrix $A$ has rank at most $n+4$. Applying B\'ezout's formula for the intersection of this with ${\mathbb P}\mathcal C(f)$, we obtain for $n=3$ the upper bound
$$2796d^3-6444d^2+3696d$$
(the Macaulay2 code is attached to the \texttt{arXiv} version of this article). There is a $\mathbb Z/2\mathbb Z\times \mathbb Z/2\mathbb Z\times \mathbb Z/2\mathbb Z$ action on the solutions. The three generators of this group act by $(u_0,u,y_0,y_1,y_2)\mapsto (u_0,-u, y_0, y_1,-y_2)$ and $(u_0,u,y_0,y_1,y_2)\mapsto (u_0,-u, -y_0, -y_1,y_2)$  and $(\lambda_0, \lambda) \mapsto (\lambda_0,-\lambda)$.
So we can divide the count by $8$. This gives the stated formula. 
\end{proof}

\bigskip
\section{Curvature of Quadrics in Three-Space}\label{sec:quadrics}
We study the case of a quadric hypersurface $X\subset \mathbb R^3$ defined by a polynomial~$f$ of degree $2$.

Let $x=(x_1,x_2,x_3)$ and write $f(x) = (x-x_0)^T A (x-x_0) - a_0$ for some symmetric matrix $A\in\mathbb R^{3\times 3}$ and $x_0\in\mathbb R^3$ and $a_0\in\mathbb R$. Since curvature is invariant under rotation and translation, we can assume that $x_0=0$, that $A=\mathrm{diag}(a_1,a_2,a_3)$ is a diagonal matrix, and $a_0=1$.
We obtain coordinates so for the general quadric hypersurface $X$:
$$
X = \{x=(x_1,x_2,x_3)\in\mathbb R^3 \mid a_1x_1^2 + a_2x_2^2 + a_3x_3^2 - 1 = 0 \}.
$$

The goal of this section is to prove the following two results for quadrics. The proofs appear in \Cref{sec_quadrics_umbilics} and \Cref{sec_cc}.
\begin{theorem}\label{thm_quadrics_umbilics}
A general quadric surface $X=\{a_1x_1^2 + a_2x^2 + a_3x_3^2 = 1\}$ has 12 complex umbilical points. The number of real umbilical points is
\begin{itemize}
\item $4$, if $X$ is an ellipsoid ($a_1,a_2,a_3$ are positive) or a two-sheeted hyperboloid (one of the $a_i$ is positive and two are negative);
\item $0$, if $X$ is a one-sheeted hyperboloid (two of the $a_i$ are positive and one is negative) or empty ($a_1,a_2,a_3$ are negative).
\end{itemize}
\end{theorem}

The number of complex umbilics was already known by Salmon; see \Cref{salmons_theorem}. The number of real umbilical points is a new result to the best of our knowledge.\enlargethispage{\baselineskip}

\begin{theorem}\label{thm_cc}
A general quadric surface $X=\{a_1x_1^2 + a_2x^2 + a_3x_3^2 = 1\}$ has $18$ complex critical curvature points. The number of real critical curvature points is
\begin{itemize}
\item $10$, if $X$ is an ellipsoid ($a_1,a_2,a_3$ are positive);
\item $4$, if $X$ is a one-sheeted hyperboloid (two of the $a_i$ are positive and one is negative);
\item $6$, if $X$ is a two-sheeted hyperboloid (one of the $a_i$ is positive and two are negative);
\item $0$, if $X$ is empty ($a_1,a_2,a_3$ are negative).
\end{itemize}
\end{theorem}

To prove the two theorems, we first set up a framework for quadric hypersurfaces in $\mathbb R^n$.
\begin{lemma}\label{def_m}
Consider a quadric hypersurface $X\subset \mathbb R^n$, defined by the quadratic polynomial $a_1x_1^2 + \cdots + a_nx_n^2 - 1$ in the variables $x=(x_1,\ldots, x_n)$ and coefficients $a=(a_1,\ldots,a_n)$. For~$x\in X$ define the polynomial
\begin{equation*}
m_x(y_1) := \sum_{j=0}^{n-1} \mu_j(x) y_1^{n-1-j},
\end{equation*}
where
$
\mu_j(x) := 2^{j-n+1} \sum_{i=1}^n s_j(a - a_ie_i)a_i^2x_i^2
$
and $s_j$ is the $j$-th symmetric polynomial in $n$ variables, $e_i$ is the $i$-th standard basis vector. Then, we have 
$$m_x(y_1)=0 \quad \Longleftrightarrow\quad \pm \, \frac{y_1}{\, 2\sqrt{(a_1^2x_1^2 + \cdots + a_n^2x_n^2)}\, } \text{ is a principal curvature of $X$ at $x$}.$$
\end{lemma}
\begin{proof}
The gradient of $f(x)=a_1x_1^2 + \cdots + a_nx_n^2- 1$ at a point $x$ is $\nabla f(x) = 2\mathrm{diag}(a)x$ and the Hessian is $\mathrm{H}f(x) = 2\mathrm{diag}(a)$.
Therefore, the curvature variety $\mathcal C(f)$ from~\eqref{def_C} for the quadric $X=\{f=0\}$ is defined by the equations
\begin{align}\label{quadric_eqns}
a_1x_1^2 + \cdots + a_nx_n^2 &= 1\\\nonumber
4\lambda^2 (a_1^2x_1^2 + \cdots + a_n^2x_n^2) &= 1\\\nonumber
u_1a_1x_1 + \cdots + u_na_nx_n &=0 \\\nonumber
u_1^2+\cdots+u_n^2&=1\\\nonumber
2\mathrm{diag}(a) u + y_1u+2y_2\mathrm{diag}(a)x &= 0.
\end{align}
When $y_2\not= 0$ we write the last equation as $$(\mathrm{diag}(a) + \tfrac{1}{2}y_1\mathbf 1_n)u =-y_2\mathrm{diag}(a)x.$$ We now multiply the $i$-th entry of this vector equation with $\prod_{j\neq i} (a_j + \tfrac{1}{2} y_1)$ and obtain

$$\big(\Pi_{j=1}^n (a_j + \tfrac{1}{2} y_1)\big)\cdot u = -y_2\, \mathrm{diag}(a_i\Pi_{j\neq i} (a_j + \tfrac{1}{2} y_1)) x.$$
Multiplying both sides by $x^T\mathrm{diag}(a)$, assuming $y_2\not=0$, and using from the third equation in \eqref{quadric_eqns} that $u^T \mathrm{diag}(a)x = u_1a_1x_1+\cdots +u_na_nx_n=0$,
we get  the polynomial equation
\begin{equation}\label{m_eq}a_1^2x_1^2\Pi_{j\neq 1} (a_j + \tfrac{1}{2} y_1)  + \cdots + a_n^2x_n^2\Pi_{j\neq n} (a_j + \tfrac{1}{2} y_1)  = 0.
\end{equation}
The left-hand side of this equation is $m_x(y_1)$.

Let us now assume that $y_2=0$.
Then, the last equation in \eqref{quadric_eqns} holds for every $u$ such that either $u_i=0$ or $y_1=-2a_i$. Let 
$$I:=\{1\leq i\leq n\mid y_1= -2a_i\}.$$ There are now two cases to consider: (1) $|I| \geq 2$. In this case, every product on the left-hand side of \eqref{m_eq} vanishes, so $m_x(y_1)=0$; (2) $|I| = 1$. Then, let $I=\{i\}$. We have $u_j=0$ for $j\neq i$. The third equation in~\eqref{quadric_eqns} becomes $u_ia_ix_i=0$. Since $u\neq 0$, we have $u_i\neq 0$, so $a_ix_i= 0$. Therefore, the left-hand side of \eqref{m_eq} also vanishes in this case, so that $m_x(y_1)=0$.

Let now $(x,u,\lambda,y)\in\mathcal C(f)$. We have shown that this implies $m_x(y_1)=0$. Furthermore, by \Cref{curv_proposition},  $g(x,u,\lambda,y)=\lambda y_1$ is a principal curvature of $X=\{f=0\}$ at $x$. For every $x\in X$, there are exactly $n-1$ principal curvatures (counted with multiplicity). Since $m_x$ has degree $n-1$, all zeros of $m_x$ actually correspond to critical curvature. This shows that, $m_x(y_1)=0$, if and only if 
$$\lambda y_1 = \pm \frac{y_1}{\Vert \nabla f(x)\Vert}  =  \pm \frac{y_1}{\Vert 2(a_1x_1,\ldots,a_nx_n)\Vert}$$
 is a principal curvature of $X$ at $x$.
\end{proof}

\subsection{Umbilical Points of Quadrics}\label{sec_quadrics_umbilics}
In this section we prove \Cref{thm_quadrics_umbilics}.
Umbilical points are points on $X$ where we have a principal curvature~$\lambda y_1$ of higher order. This means that at $x\in X$ the polynomial $m_x$ from \Cref{def_m} has a zero of order at least two.

We use \Cref{def_m} in the case $n=3$; i.e., quadrics in three-space. The coefficients of the polynomial $$m_x(y_1) = \mu_0(x)y_1^2 + \mu_1(x)y_1 + \mu_2(x)$$
are in this case
\begin{equation}
\begin{aligned}
\label{mus}
\mu_0(x) &= \tfrac{1}{4}(a_1^2x_1^2 + a_2^2x_2^2 + a_3^2x_3^2),\\
\mu_1(x) &= \tfrac{1}{2}(a_1^2x_1^2 (a_2+a_3) + a_2^2x_2^2 (a_1+a_3) + a_3^2x_3^2(a_1+a_2) ),\\
\mu_2(x) &= a_1^2x_1^2 a_2a_3 + a_2^2x_2^2 a_1a_3 + a_3^2x_3^2a_1a_2,
\end{aligned}
\end{equation}
and we can use the discriminant for quadratic polynomials to see that umbilics on the quadric surface $X = \{a_1x_1^2 + a_2x_2^2 + a_3x_3^2 = 1\}$ are points such that $$\mu_1(x)^2 - 4\mu_0(x)\mu_2(x) = 0.$$ Using~\eqref{mus} we can write this explicitly as
\begin{equation}
\begin{aligned}\label{disc_umbilical}
\mu_1(x)^2 - 4\mu_0(x)\mu_2(x) \ = \ & \frac{1}{4}a_1^4(a_2-a_3)^2x_1^4 + \frac{1}{4}a_2^4(a_1-a_3)^2x_2^4 + \frac{1}{4}a_3^4(a_1-a_2)^2x_3^4\\
& - \frac{1}{2}a_1^2a_2^2(a_3-a_1)(a_2-a_3) x_1^2x_2^2\\
& - \frac{1}{2}a_1^2a_3^2(a_1-a_2)(a_2-a_3) x_1^2x_3^2\\
& - \frac{1}{2}a_2^2a_3^2(a_1-a_2)(a_3-a_1) x_2^2x_3^2.
\end{aligned}
\end{equation}
This polynomial factors completely over the complex numbers:
\begin{align*}
\mu_1(x)^2 - 4\mu_0(x)\mu_2(x) \ = \ &
 \frac{1}{4}(a_1\sqrt{a_2-a_3}\cdot x_1+a_2\sqrt{a_3-a_1}\cdot x_2+a_3\sqrt{a_1-a_2}\cdot x_3)\\
&\cdot (a_1\sqrt{a_2-a_3}\cdot x_1-a_2\sqrt{a_3-a_1}\cdot x_2+a_3\sqrt{a_1-a_2}\cdot x_3)\\
&\cdot (a_1\sqrt{a_2-a_3}\cdot x_1+a_2\sqrt{a_3-a_1}\cdot x_2-a_3\sqrt{a_1-a_2}\cdot x_3)\\
&\cdot (a_1\sqrt{a_2-a_3}\cdot x_1-a_2\sqrt{a_3-a_1}\cdot x_2-a_3\sqrt{a_1-a_2}\cdot x_3).
\end{align*}
The pairwise common zeros of the four factors are the following three pairs of lines through the origin:
\begin{equation}
\begin{aligned}\label{lines}
 x_1 &= a_2^2(a_3-a_1)x_2^2-a_3^2(a_1-a_2)x_3^2=0 \\ 
x_2 &= a_1^2(a_2-a_3)x_1^2-a_3^2(a_1-a_2)x_3^2=0 \\ 
x_3 &= a_1^2(a_2-a_3)x_1^2-a_2^2(a_3-a_1)x_2^2=0
\end{aligned}
\end{equation}
of which $2$ are real.  These two lines may or may not intersect the quadric in real points, so there are $0,2$ or $4$ real umbilical points among the $12$ complex intersection points of these lines with the complex quadric surface.

To determine when we have which number of real umbilical points, without loss of generality, we can assume that the coefficients are ordered as $a_1\leq a_2\leq a_3$. We can also assume that $a_3>0$, since otherwise the real locus of the quadric is empty. We can write \eqref{disc_umbilical} as a \emph{sum-of-squares}:
\begin{equation}
\begin{aligned}\label{disc_umbilical2}
\mu_1(x)^2 - 4\mu_0(x)\mu_2(x) \ = \ & \frac{1}{4}(a_1^2(a_2-a_3)x_1^2 + a_3^2(a_2-a_1)x_3^2)^2 \\
& + \frac{1}{4}a_2^4(a_1-a_3)^2x_2^4\\
& + \frac{1}{2}a_1^2a_2^3(a_3-a_1)(a_3-a_2) x_1^2x_2^2\\
& + \frac{1}{2}a_2^2a_3^3(a_1-a_2)(a_1-a_3) x_2^2x_3^2.
\end{aligned}
\end{equation}
\begin{remark}
The fact that \eqref{disc_umbilical2} is a sum-of-squares shows that real umbilical points correspond to singular solutions of $\mu_1(x)^2 - 4\mu_0(x)\mu_2(x)=0$.
\end{remark}

We can see from \eqref{disc_umbilical2} that if one of the $a_i$ is zero or if two of the $a_i$ are equal, then we get infinitely many umbilical points. In all other cases, we get finitely many real umbilical points as follows: from the first and second equations in \eqref{lines} we have
$$x_2=0\quad \text{and}\quad x_1 =\pm   \frac{a_3}{a_1}\sqrt{\frac{a_2-a_1}{a_3-a_2}} \, x_3.$$ Note that the expression under the square root is positive by assumption. Plugging this into $a_1x_1^2 + a_2x_2^2 + a_3x_3^2 = 1$ gives
$\left(\frac{a_3^2}{a_1}\frac{a_2-a_1}{a_3-a_1} +a_3\right) x_3^2 = 1.$
Therefore, we have four real umbilical point if and only if $$\frac{a_3^2}{a_1}\, \frac{a_2-a_1}{a_3-a_2} + a_3>0\quad\Longleftrightarrow\quad\frac{a_3^2(a_2-a_1) + a_3a_1(a_3-a_2)}{a_1(a_3-a_2)} = \frac{a_2a_3(a_3 - a_1)}{a_1(a_3-a_2)}  >0.$$
This is the case if and only if $a_1,a_2,a_3>0$ or $a_3>0$ and $a_1,a_2<0$.

\subsection{Critical Curvature Points of Quadrics}\label{sec_cc}
In this section, we prove \Cref{thm_cc}. Critical curvature points are points on $X$ where one of the principal curvatures is critical.
Recall from \Cref{def_points} the following algebraic formulation: there exist $\lambda, u, y$ with $(x,\lambda, u, y)\in\mathcal C(f)$, the curvature variety, such that $(x,\lambda, u, y)$ is a critical point of $g$.
In the general setting, where $X\subset \mathbb R^n$ we eliminate $\lambda$ and~$u$ using the following proposition. The formulation of the next proposition uses the notation $\eta(x)= \langle \nabla f(x), \nabla f(x)\rangle $ that we have established in \Cref{def_eta}.
\begin{proposition}\label{prop_equations_cc}
Consider the quadric hypersurface $X\subset \mathbb R^n$ defined by the polynomial $f(x)=a_1x_1^2 + \cdots + a_nx_n^2-1$. For $(x,\lambda, u, y)\in\mathcal C(f)$ let
$$q(x,y_1):=8\,y_1 \,m_x'(y_1) = 8\sum_{j=0}^{n-2}(n-j-2)\mu_j(x)y_1^{n-1-j}$$
(here, $m_x'(y_1)$ means the derivative of the polynomial $m_x(y_1)$ in $y_1$).
The $x$-coordinates of the zeros of the system of $n+3$ polynomial equations in the $n+3$ variables $(x_1,\ldots,x_n,\lambda,y_1,t)$
$$
\begin{bmatrix}
f(x)\\[0.4em]
m_x(y_1)\\[0.4em]
\lambda^2\eta(x)-1\\[0.4em]
\eta(x)\sum_{j=0}^{n-1}  y_1^{n-1-j} \, \nabla  \mu_j(x) +(t\mathbf 1_n +  q(x,y_1)\mathrm{diag}(a)) \mathrm{diag}(a)x
\end{bmatrix} =0.
$$
are exactly the critical curvature points of $X$.
\end{proposition}
We prove the proposition at the end of this section. Let us first see how to use it to prove \Cref{thm_cc}.
In the case $n=3$, \Cref{prop_equations_cc} provides 6 polynomial equations in the 6 variables $(x_1,x_2,x_3,\lambda,y_1,t)$.
Critical curvature points are points $x=(x_1,x_2,x_3)\in\mathbb C^3$ for which this set of $6$ polynomial equations has a solution $(\lambda,y_1,t)$.

Eliminating $\lambda$ from these equations is straightforward.
But eliminating $y_1$ and $t$ as in the above case seems intractable. With the following lemma, we can eliminate immediately one of the $x_i$.

\begin{lemma}\label{distinctai}
A quadric surface $X=\{a_1x_1^2 + a_2x_2^2 + a_3x_3^2 = 1 \}$ with $a_1,a_2,a_3\not= 0$ and pairwise distinct has critical curvature points only in coordinate planes.
\end{lemma}
\begin{proof}
We have
$\nabla  \mu_j(x) = \frac{1}{2^{1-j}}D_jx,$
where $$D_j=
\begin{bmatrix}
s_j(a-a_1e_1)a_1^2 & 0 & 0 \\
0 & s_j(a-a_2e_2)a_2^2 & 0\\
0 & 0 & s_j(a-a_3e_3)a_3^2
\end{bmatrix}
.$$
We also denote
$$M:=\begin{bmatrix}
ta_1+q(x,y_1)a_1^2 & 0 & 0 \\
0 & ta_2+q(x,y_1)a_2^2 & 0\\
0 & 0 & ta_3+q(x,y_1)a_3^2
\end{bmatrix}.$$
Let us use this to write the last $n=3$ equations in \Cref{prop_equations_cc} as follows:
\begin{equation}\label{distinctai_eq1}
\Big(\eta(x)\big( \tfrac{1}{2}\,y_1^2\,D_0+y_1 \, D_1+ 2\,D_2\big)   +
  M\Big)x = 0.
\end{equation}
We assume by way of contradiction that $x_1,x_2,x_3\not=0$. Then, in \eqref{distinctai_eq1} we can divide the $i$th equation by $x_i$. Furthermore, since $a_i\neq 0$ for $1\leq i\leq 3$  by assumption, we can also divide the $i$th equation by $a_i$. We get the following matrix equation:
\begin{equation}\label{eq_distinctai_1}\begin{bmatrix}
 a_1(a_2+a_3) &a_1a_2a_3& 1\\
a_2(a_1+a_3) &a_1a_2a_3& 1\\
a_3(a_1+a_2) &a_1a_2a_3& 1
\end{bmatrix}
\cdot
\begin{bmatrix}
\eta(x)y_1\\
2\eta(x)\\
t\\
\end{bmatrix}=-(q(x,y_1) + \tfrac{1}{2}\eta(x)y_1^2)\begin{bmatrix}
a_1\\
a_2\\
a_3
\end{bmatrix}.
\end{equation}
Let $C$ denote the $3\times 3$-matrix in this equation.
It has rank $2$. The right kernel of $C$ is
$$C\begin{bmatrix} 0 \\ 1 \\ -a_1a_2a_3\end{bmatrix} = 0.$$
We have $\eta(x)\neq 0$ because $\lambda^2\eta(x)=1$, and we have $y_1\neq 0$ because otherwise the quadric would be a cone.
Since the rank of $C$ is $2$, this shows that the right-hand side of \eqref{eq_distinctai_1} must not be zero. We also have the left kernel $$\begin{bmatrix}a_1(-a_2+a_3)& a_2(a_1-a_3)& a_3(-a_1+a_2)\end{bmatrix}\cdot C=0.$$
The matrix equations therefore have solutions only if
\[
\begin{bmatrix}a_1(-a_2+a_3)& a_2(a_1-a_3)& a_3(-a_1+a_2)\end{bmatrix}\cdot \begin{bmatrix}
a_1\\
a_2\\
a_3
\end{bmatrix}=2(a_1-a_2)(a_2-a_3)(a_1-a_3)=0.
\]
We have a contradiction and the lemma follows. \end{proof}

By \Cref{distinctai}, we need only to
consider possible critical points in coordinate planes.  We shall see that when $a_1,a_2,a_3$ are pairwise distinct real numbers, there are $18$ complex solutions to the critical curvature equations in the coordinate planes.

In fact, we may assume $x_3=0$, in which case the number of equations in Proposition \ref{prop_equations_cc} reduces to four.  Using Macaulay2 \cite{M2},
we find the minimal primes of the ideal generated by the four polynomials. The minimal primes contain one of the following sets of generators:

\begin{align*}
&(a_1),\; (a_2),\; (a_1-a_2),\; (x_1),\; (x_2),\;(y_1), \text{ and }
(x_2^2a_2a_3(a_1-a_2)+a_1(a_2-a_3))
\end{align*}
Of these the first two define cylinders, and the third a rotated conic section, all three with infinitely many critical points.  The fourth and fifth define the points of the quadric that lie on the coordinate axes.  These are solutions to the critical  curvature equations.  By symmetry, there are $6$ complex solutions on the coordinate axes.
The sixth set of generators contains~$y_1$.  The eigenvalue $y_1=0$ appears when a principal curvature vanishes, cf. Proposition 2.5, which happens only when the quadric $X$ is a cone.

Finally, for the seventh set of generators we get
$x_2^2  = a_1(a_3-a_2)/(a_2a_3(a_1-a_2)).$
The solution in $x_1^2$ is given by symmetry
This yields four complex solutions.
So there are $12$ complex solutions total in the three coordinate planes, outside the coordinate axes.

Assume now that $a_1<a_2<a_3$.
Only the four solutions in $x_2=0$ are real if $a_1>0$.
If $a_1<0<a_2$, no solutions are real.
If $a_2<0<a_3$, then four solutions are real in the plane given by $x_2=0$.

We summarize:  There are 18 complex critical curvature points on the quadric surface $a_1x_1^2+a_2x_2^2+a_3x_3^2=1$ with $a_1<a_2<a_3$, if $a_1,a_2,a_3\neq 0$ and pairwise distinct. All 18 critical curvature points are in the coordinate planes. The real critical curvature points are:
\begin{itemize}
\item Ellipsoid,  $a_1>0$:

$6$ points $(\pm 1/\sqrt{a_1},0,0),(0,\pm 1/\sqrt{a_2},0), (0,0,\pm 1/\sqrt{a_3})$ and $4$ points in $\{x_2=0\}$.
\item One-sheeted hyperboloid $a_1<0<a_2$:

$4$ points on coordinate axes and none outside of coordinate planes.
\item Two-sheeted hyperboloid, $a_2<0<a_3$:

$2$ points on coordinate axes and $4$ points in~$\{x_2=0\}$.
\end{itemize}

Finally, when two or three $a_i$ are equal, there are infinitely many critical curvature points because the quadric in this case contains circles. When one $a_i=0$, then there are also infinitely many critical curvature points because the quadric is a cone in this case. This proves \Cref{thm_cc}.

It remains to prove \Cref{prop_equations_cc}.

\begin{proof}[Proof of \Cref{prop_equations_cc}]
We know from
\Cref{curv_proposition} that the curvature at a point $x\in X$ can be written as $g(x,\lambda,u,y)=-\lambda y_1$. Suppose that $x$ is not an umbilical point of $X$. Then~$y_1$ is a simple zero of the equation $m_x(y_1)=0$ (see \Cref{def_m}), so locally $y_1=y_1(x)$ is a differentiable function of $x$.

By \Cref{def_m}, the function $m_x(y_1(x)) = \sum_{j=0}^{n-1} \mu_j(x) y_1(x)^{n-1-j}=0$ vanishes in a neighborhood of~$x\in X$. Therefore, the derivative of $m_x(y_1(x))$ with respect to $x$ must be in the normal space of $X$ at $x$. We get
\begin{equation*}
 \sum_{j=0}^{n-1}  y_1^{n-1-j}\,\nabla\mu_j(x) + m_x'(y_1)\,\nabla y_1(x)  = t_1\mathrm{diag}(a)x \quad \text{for some $t_1\in\mathbb R$}.
\end{equation*}
Next, we use that also $\lambda=\lambda(x)$ is locally a function of $x$, because $\lambda^2  \eta(x) = 1$. We have $$\eta(x) \nabla \lambda(x) + \lambda(x) \nabla \eta(x) = \eta(x)\nabla \lambda(x) + 8\lambda \mathrm{diag}(a)^2x= 0.$$
The curvature $g(x)=-\lambda(x) y_1(x)$ is locally a differentiable function of~$x$. Its derivative is
\begin{equation}\label{cc_eq11}
-\nabla g(x) = \lambda(x) \nabla y_1(x) + y_1(x) \nabla \lambda(x).
\end{equation}
In the following, we drop the notation for a term being a function of $x$ or $y_1$. 

Recall that $q=8\,y_1\,m_x'$.
Multiplying both sides of \eqref{cc_eq11} with $\eta \, m_x'$ we get the following polynomial equations:
\begin{equation}\label{cc_eq1}-(\eta\,m_x') \nabla g  =\lambda\Big(\eta\, t_1\,\mathrm{diag}(a)\,x -\eta\sum_{j=0}^{n-1}  y_1^{n-1-j}\,\nabla\mu_j -   q\, \mathrm{diag}(a)^2\,x\Big).
\end{equation}
These polynomial equations hold for all $x\in X$ that are not umbilical points. By continuity, they must hold on all of $X$.

We claim that $x$ is a critical curvature point if and only if
\begin{equation}\label{cc_eq2}
\lambda\Big(\eta\, t_1\,\mathrm{diag}(a)\,x -\eta\sum_{j=0}^{n-1} y_1^{n-1-j}\,\nabla\mu_j - q\, \mathrm{diag}(a)^2\,x\Big) = t_2\, \mathrm{diag}(a)\,x \quad \text{for some $t_2\in\mathbb R$.}
\end{equation}
By \Cref{umbilical_are_cc}, umbilical points are critical curvature points. We have two cases: First, if $x$ is an umbilical point, the left-hand side of \eqref{cc_eq1} vanishes, so that we can take $t_2=0$ in this case. Second, if $x$ is not an umbilical point, we have $m_x'(y_1)\neq 0$ (and we have $\eta\neq 0$ since $\lambda^2\eta=1$), so that \eqref{cc_eq2} in this case is equivalent to $\nabla g \in \mathbb R \cdot \nabla f$ by \eqref{cc_eq1}.
Now, we substitute $t_2 - \lambda\,\eta\, t_1$ by $\lambda t$, divide by $\lambda\neq 0$, and finally obtain
\begin{equation}\label{prop_equations_cc_eq1}
 \eta\sum_{j=0}^{n-1}
 y_1^{n-1-j}\,\nabla\mu_j  +(t\,\mathbf 1_n +  q\,\mathrm{diag}(a)) \mathrm{diag}(a)\,x =0.
\end{equation}
Together with the equation for $X$, namely $a_1x_1^2 + a_2x_2^2+a_3x_3^2=1$, the equation $m_x(y_1)=0$ that defines $y_1$ to be the curvature of $X$ at $x$, and $\lambda^2\eta(x)=1$ (which we use to have $\lambda\neq 0$ and $\eta(x)\neq 0$) the equations in \eqref{prop_equations_cc_eq1} define umbilical and critical curvature points on $X$.
\end{proof}

\bigskip
\section{Acknowledgments.}
M.\ Weinstein was supported by the National Science Foundation Mathematical Sciences Postdoctoral Research Fellowship. This material is based upon work supported by the National Science Foundation under Award No. 2103261.

P.\ Breiding is funded by the Deutsche Forschungsgemeinschaft (DFG, German Research Foundation) -- Projektnummer 445466444.

This project was initiated when M.\ Weinstein was visiting the Thematic Einstein Semester on
Algebraic Geometry: ``Varieties, Polyhedra, Computation'' in the Berlin Mathematics Research Center MATH+.\enlargethispage{\baselineskip}

We thank two anonymous referees for their feedback, which helped us improving the paper.
\bigskip

\noindent
{\bf Authors' addresses:}

\smallskip

\noindent Paul Breiding,
 \ University of Osnabr\"uck \hfill  {\tt pbreiding@uni-osnabrueck.de}

\noindent Kristian Ranestad, University of Oslo
\hfill {\tt ranestad@math.uio.no}

\noindent Madeleine Weinstein,
University of Puget Sound \hfill {\tt mweinstein@pugetsound.edu}

\end{document}